\documentclass[11pt]{article}

\usepackage[square,numbers]{natbib}

\usepackage{color,latexsym,amsthm,amsmath,amssymb,path,enumitem,url,bbm}
\usepackage{physics}

\newcommand{\ignore}[1]{}

\newtheorem{theorem}{Theorem}[section]
\newtheorem{lemma}[theorem]{Lemma}
\newtheorem{corollary}[theorem]{Corollary}

\newcommand{\Proof}[1]
        {
        \noindent
        \emph{Proof #1.}~
        }
\newsavebox{\smallProofsym}                     
\savebox{\smallProofsym}                        %
        {
        \begin{picture}(7,7)                    %
        \put(0,0){\framebox(6,6){}}             %
        \put(0,2){\framebox(4,4){}}             %
        \end{picture}                           %
        }                                       %
\newcommand{\smalleop}[1]
        {
        \mbox{} \hfill #1~~\usebox{\smallProofsym}\!\!\!\!\!\!\
        }

\newcommand{\parag}[1]{\vspace{2mm}

\noindent{\bf #1} }

\usepackage{graphicx,psfrag}
\usepackage[rflt]{floatflt}

\usepackage{dsfont}

\newcommand{\RR}{\ensuremath{\mathbb R}}

\newcommand{\FF}{\ensuremath{\mathbb F}}

\newcommand{\pts}{\mathcal P}

\newcommand{\lines}{\mathcal L}

\begin{document}
\pagenumbering{arabic}

\title{Sharp Szemer\'edi-Trotter Constructions in the Plane}

\author{
Larry Guth\thanks{Department of Mathematics, MIT, Cambridge, MA 02139.
{\sl lguth@math.mit.edu}. Supported by Simons Investigator Award}
\and
Olivine Silier\thanks{Caltech, Pasadena, CA 91125.
{\sl osilier@caltech.edu.} Supported by the Associates SURF Fellowship}}

\maketitle

\begin{abstract}
We present a new family of sharp examples for the Szemer\'edi-Trotter theorem.  These are the first examples not based on a rectangular lattice.  We also include an application to the discrete inverse Loomis-Whitney problem.
\end{abstract}

\section{Introduction}

One formulation of the celebrated Szemer\'edi--Trotter theorem \cite{ST83} provides a tight upper bound for the number of $r$-rich lines:

\begin{theorem}{\bf (Szemer\'edi and Trotter)} \label{th:dualSzemTrot}
Let $\pts$ be a set of $n$ points and let $\lines_r$ be a set of lines we call \textit{$r$-rich} that contain at least $r$ points in $\pts$, both in $\RR^2$. Then
\[ |L_r| = O\qty(\frac{n^2}{r^3} + \frac{n}{r}). \]
\end{theorem}

This statement is equivalent to the statement in terms of point-line incidences, which goes as follows.

\begin{theorem}{\bf (Szemer\'edi and Trotter)} \label{th:SzemTrot}
Let $\pts$ be a set of $n$ points and let $\lines$ be a set of $m$ lines, both in $\RR^2$. Then
\[ I(\pts,\lines) = O(m^{2/3}n^{2/3}+m+n). \]
\end{theorem}

The many variants of this theorem constitute an entire discipline called incidence theory. The theorem has also proved useful in other domains: its numerous applications range from problems in additive number theory to harmonic analysis  \cite{bombieri2015problem,BD15,CGS16,katz2019improved,singer2021point}. Despite the community's interest \cite{ABCP,Elekes97b,Elek02,Solymosi06}, the \textit{inverse problem} : characterizing constructions that meet the Szemeredi-Trotter (mixed term) upper bound, remains widely open. Although there has been recent progress for lines in general position on Cartesian product point sets, \cite{CL,Murphy} not much else is known. 

Up to recently, only two constructions were known to match the (non-linear) term in Theorem \ref{th:dualSzemTrot} and Theorem \ref{th:SzemTrot}.  The first example, given by Erd\H{o}s in 1946, is based on a square lattice.  The second example, given by Elekes \cite{Elek02} in 2001, is based on a rectangular lattice. Adam Sheffer and the second author recently introduced the first infinite family of sharp Szemeredi-Trotter examples, which has the Erd\H{o}s and Elekes constructions as limits \cite{ShefferSilier}. In all of these examples, the point set is a \textit{lattice}: a Cartesian product of two arithmetic progressions.

Every previous sharp example for the Szemer\'edi-Trotter theorem was found by starting with a Cartesian product of two arithmetic progressions and then applying a projective transformation and/or point-line duality.  In this paper, we give a new sharp example which does not have this structure.

\parag{Our family of constructions.}

We present the first sharp Szemer\`edi-Trotter family of non-lattice point-line constructions in $\RR^2$: the $x$ and $y$ coordinates of the point set are a generalized arithmetic progression and for any richness $r$ there is a maximal family of $r$-rich lines on the point set. 

\begin{theorem} \label{th:construction}
For any non-square integer $k$, any large enough $N$ and $r \leq N$, let the point set $\pts=A_N^2$ where $A_N = \qty{ x_1 + x_2 \sqrt{k} ; x_1,x_2 \in \qty[-\sqrt{N},\sqrt{N}]}$. Then there exists a set of $r$-rich lines $|\lines_r|$ such that 
\[ |\lines_r| = \Theta\qty(\frac{|\pts|^2}{r^3} + \frac{|\pts|}{r}). \]
\end{theorem}

See Section  \ref{sec:construction} for the proof and the explicit construction of the line set. 
Our point set $\pts = A_N^2$ is not a product of arithmetic progressions.  It is a product of generalized arithmetic progressions.  However, not every product of generalized arithmetic progressions gives a sharp example for Szemer\'edi-Trotter.  The algebraic structure coming from $\sqrt{k}$ is crucial.  If we replace $\sqrt{k}$ by a transcendental number, then the construction would be far from sharp for Szemer\'edi-Trotter.  


\parag{Application to Inverse Discrete Loomis-Whitney.}

The Loomis-Whitney inequality \cite{LW} upperbounds the volume of an $n$ dimensional set by the product of the areas of its "shadows": the $n-1$ dimensional coordinate projections. 

\begin{theorem} \label{th:LoomisWhitney}
Let $m$ be the measure of an open subset $O$ of the Euclidean $n$-space, and let $m_1,...,m_n$ be the $(n-1)$-dimensional measures of the projections of $O$ on the coordinate hyperplanes. Then  
\[ m^{n-1} \leq \Pi_{i=1}^n m_i. \]
\end{theorem}

The many variations of this theorem constitute a rich field of study \cite{Ball,BT,BCW,Do20}. These results also find applications in other domains from group theory \cite{gromov} to the Kakeya problem in harmonic analysis \cite{BCT06}.
Recently there has been much interest in the \textit{inverse problem}: characterizing sets that provide sharp examples of the Loomis-Whitney inequality \cite{CGG18,AB21}.
We focus on the discrete variant of the inverse problem in $\RR^2$: characterizing point configurations in the plane whose 1 dimensional projections are minimal. Classical Loomis-Whitney tells us that in the case of a point set in $\RR^2$ of size $n^2$ (using affine transformations to map 2 arbitrary projection directions to the coordinate projections) the product of the size of these two projections is greater or equal to $n^2$. Equivalently, it is not possible for both projections to have size less than $n$. 

Thus the natural inverse discrete Loomis-Whitney problem in the plane asks under which structural conditions of the point set of size $n^2$, and for which set of projection directions, all the one-dimensional projections have size $\Theta(n)$. Elementary arguments yield the following necessary and sufficient condition for square lattices:

\begin{lemma}\label{th:LM_lattice}
Let the point set $\pts$ be a section of the integer lattice of size $n \times n$. A one-dimensional projection  of $\pts$ has size $\Theta(n)$ if and only if the slope of the projection direction is an irreducible rational $p/q$ such that $p,q=O(1)$.
\end{lemma}

Note for any square lattice in the plane there exists an affine map which takes it to a square section of the integer square lattice. So up to affine transformation of the plane lemma \ref{th:LM_lattice} holds for any square lattice.

Obtaining sharp constructions for the discrete inverse Loomis-Whitney problem in the plane for an $n \times n$ grid of points overlaps with finding sharp examples for Theorem \ref{th:dualSzemTrot} because finding a family of $\Theta(n)$ parallel $\Theta(n)$-rich lines yields a projection direction along which a constant fraction of the points have minimal projection size. We obtain the following application of theorem \ref{th:construction}:

\begin{corollary}\label{th:LM_GAP}
For any non-square integer $k$, any large enough $N$, let the point set $\pts=A_N^2$ where $A_N = \qty{ x_1 + x_2 \sqrt{k} ; x_1,x_2 \in \qty[-\sqrt{N},\sqrt{N}]}$. Then for any constant $p=O(1)$ there is a set of projections $\big\{ \pi_i \}_{i=1}^{\Theta(p)}$ such that $|\pi_i(\pts) | = \Theta \qty( \sqrt{p} n)$.
\end{corollary}

\parag{Sharp example for Energy Bound.}

Our constructions provide a new tight example for the following lemma which provide upperbounds for the additive energy of finite subsets of $\RR$ \cite{MRS13}: 

\begin{lemma}\label{th:energy_bounds}
Let $A$, $B$ and $X$ be finite subsets of $\RR$ such that $|X| \leq |A| |B|$. Then
\[ \sum_{x \in X}  E^+(A,xB) = O\qty( |A|^{3/2} |B|^{3/2} |X|^{1/2}). \]
\end{lemma}

Note there is an equivalent lemma for multiplicative energy. \cite{MRS13} The proofs of this lemma relies on an application of Theorem \ref{th:dualSzemTrot} so all of our sharp examples from Theorem \ref{th:construction} are also sharp for this lemma.

\begin{lemma}\label{th:energy_ex}
For any non-square constant $k$, and constant $M \leq N$ let 

\noindent $A_N = \qty{ x_1 + x_2 \sqrt{k} ; x_1,x_2 \in \qty[-\sqrt{N},\sqrt{N}]}$ and let $X=S \subset \frac{A_N}{A_N}$ be the slope set from the proof of Theorem \ref{th:construction}. Then $|X| \leq |A_N|^2$ and
\[ \sum_{x \in X}  E^+(A_N,xA_N) = O\qty( |A_N|^3 |X|^{1/2}). \]
\end{lemma}

Note we can construct an equivalent sharp example for the multiplicative energy version of the lemma. 

\parag{Acknowledgments.}
The authors are grateful to Lingxian Zhang for interesting conversations.

\section{Background}

This section presents tools that will be used in the proofs. 

\textbf{Asymptotic notation} is used throughout. We say $f(n) = O(g(n))$ if there exist constants $c, n_0 >0$ such that $|f(n)| \leq c \cdot g(n)$ for all $n \geq n_0$. Likewise $f(n) = \Omega(g(n))$ if there exist constants $c, n_0 >0$ such that $|f(n)| \geq c \cdot g(n)$ for all $n \geq n_0$. We say $f(n) = \Theta(g(n))$ if and only if $f(n) = O(g(n))$ and $f(n) = \Omega(g(n))$.

We also use the stronger notation $f(n) = o(g(n))$ if for all $\epsilon>0$ there exists $n_\epsilon$ such that $|f(n)| \leq \epsilon \cdot g(n)$ for all $n \geq n_\epsilon$. Likewise we say $f(n) = \omega(g(n))$ if for all $\epsilon>0$ there exists $n_\epsilon$ such that $|f(n)| \geq \epsilon \cdot g(n)$ for all $n \geq n_\epsilon$.

The following classical result from Beck \cite{Beck83} is used in the proof of Theorem \ref{th:construction}.

\begin{theorem} \label{th:Beck}{\bf (Beck)}
There exist constants $c$, $k$ such that for any set of $n$ points in $\RR^2$
\begin{itemize}
    \item \textbf{either} there is a subset of $n/c$ colinear points
    \item \textbf{or} there are $\Omega(n^2/k)$ distinct lines containing at least two points of the point set. Such lines are said to be \textit{determined} by the point set.
\end{itemize}

\end{theorem}

\section{New constructions} \label{sec:construction}

In this section we prove Theorem \ref{th:construction} and provide an explicit description of the line set. We first recall the statement of the theorem.
\vspace{2mm}

\noindent {\bf Theorem \ref{th:construction}.}$\qquad$\\
\emph{For any non-square integer $k$, any large enough $N$ and $r \leq N$, let the point set $\pts=A_N^2$ where $A_N = \qty{ x_1 + x_2 \sqrt{k} ; x_1,x_2 \in \qty[-\sqrt{N},\sqrt{N}]}$. Then there exists a set of $r$-rich lines $|\lines_r|$ such that 
\[ |\lines_r| = \Omega \qty(\frac{\pts^2}{r^3} + \frac{\pts}{r}). \]
}

\begin{proof}
We double count the number of incidences to obtain a lower bound on the size of the line set. This involves proving each of the lines are $r$-rich. Let $M=\frac{N}{r}$.

We define the point set, slope set and line set as follows:

\begin{equation*}
    A_N = \qty{ x_1 + x_2 \sqrt{k} \hspace{3pt}|\hspace{3pt} x_1,x_2 \in [-\sqrt{N},\sqrt{N}]}
\end{equation*}
There are $2\sqrt{N}$ choices of $x_i$ and $k$ is not a square so $|A_N| = \qty(2\sqrt{N})^2 = 4N = \Theta(N)$. Letting $P=A_N^2$ we have $|P|=\Theta\qty(N^2)$. Next I define the slope set

\begin{equation*}
    S = \qty{\frac{p_1+p_2\sqrt{k}}{q_1+q_2\sqrt{k}} ;  |p_i|,|q_i| \in \qty[c \sqrt{M},\sqrt{M}] , \gcd(p_1^2-k p_2^2,q_1^2-k q_2^2) \leq 5, \gcd(p_1,p_2) \leq 5}
\end{equation*}
for some constant $c < 1$ sufficiently close to 1. Then let the line set 

\begin{equation*}
    L = \qty{y=s(x-a)+b ; (a,b) \in A_{N/4}^2,s\in S}
\end{equation*}

Each point $(a,b)\in A_{N/4}^2 \subset P$ has at least $|S|$ lines of $L$ so $I(P,L)=\Omega\qty(|A_{N/4}^2| |S|)=\Omega\qty(N^2 |S|)$.

\begin{lemma}\label{th:slopes}
$|S|=\Theta\qty(M^2)$
\end{lemma}

\begin{proof}
We represent the slopes in $S$ as points in the plane where the $x$-coordinate is the numerator and the $y$-coordinate is the denominator. Let $S_P=\Big\{ (p_1+p_2\sqrt{k},q_1+q_2\sqrt{k}) \in A_M^2 \text{ where } |p_i|,|q_i| \in \qty[c \sqrt{M},\sqrt{M}]$, $\gcd(p_1^2-kp^2,q_1^2-kq_2^2) \leq 5, \gcd(p_1,p_2) \leq 5 \Big\} \subset\RR^2$. Note the number of distinct elements in the slope set $S$ is equal to the number of lines determined by a point in $S_P$ and the origin. Let $S_P^+ \subset S_P$ be the subset of elements where $p_i,q_i \geq 0$, then $|S_P|=16|S_P^+|$. 

No line can contain more than $\sqrt{|S_P^+|}$ points of $S_P^+$ since each line contains at most one point per row/column. Thus we must be in the second case of Theorem \ref{th:Beck}: there exists a constant $K$ such that at least $|S_P^+|^2/K$ distinct lines are determined by points in $S_P^+$. So there exists $\alpha=(\alpha_x,\alpha_y) \in S_P^+$ such that the set $L_\alpha$ of lines determined by $\alpha$ and another point in $S_P^+$ satisfies $|L_\alpha| \geq |S_P^+|/K$. 

For all $l \in L_\alpha$ there exists a point $\beta = (\beta_x,\beta_y)  \in S_P^+ \backslash \{\alpha \}$ such that $\beta \in l$. Letting $l : y-\alpha_y= s(x -\alpha_x) $ we must have $(\beta_y - \alpha_y) = s(\beta_x -\alpha_x) $. $\alpha, \beta \in S_P^+$ so $\beta -\alpha \in S_P$. Thus the line of slope $s$ going through the origin also contains the point $\beta - \alpha \in S_P$. Finally all the lines in $L_\alpha$ are concurrent and distinct so none of them have the same slope. Thus $|S| \geq |L_\alpha|\geq |S_P^+|/K $. 

To find a lower bound on $S_P^+$, we must remove the quadruples $(p_1,p_2,q_1,q_2)$ that do not satisfy the divisibility requirements. The number of quadruples $(p_1,p_2,q_1,q_2) \in \qty[c \sqrt{M},\sqrt{M}]$ such that $\gcd(p_1^2-kp_2^2,q_1^2-kq_2^2) \leq 5$ and $\gcd(p_1,p_2) \leq 5$ is equal to $(1-c)^2 M^2$ minus the number of quadruples such that $d \mid p_1^2-kp_2^2$ and $d \mid q_1^2-kq_2^2$ for some odd prime $d > 5$ and the number of quadruples where $d \mid p_1$ and $d \mid p_2$ for some odd prime $d > 5$. 

We first count the number of quadruples such that $d \mid p_1^2-kp_2^2$ and $d \mid q_1^2-kq_2^2$. If $kp_2^2$ is a quadratic residue mod $d$ (prime) then $\FF_d$ is a field so the degree 2 equation for $p_1$ in $\FF_d: p_1^2=kp_2^2 \mod d$ has at most two solutions. Likewise for $q_1$, so we must remove at most $( (1-c)\sqrt{M} \cdot 2(1-c)\sqrt{M}/d)^2 = \frac{4(1-c)^2 M^2}{d^2}$ quadruples $(p_1,p_2,q_1,q_2)$. Furthermore the number of quadruples $(p_1,p_2,q_1,q_2)$ such that $d \mid p_1$ and $d \mid p_2$ for some odd prime $d$ which we must remove is upperbounded by $((1-c)\sqrt{M})^2 \cdot ((1-c)\sqrt{M}/d)^2 = \frac{(1-c)^4 M^2}{d^2}$. 

So for each $d$ the number of quadruples we must remove is upperbounded by $\frac{4(1-c)^2 M^2}{d^2} + \frac{(1-c)^4 M^2}{d^2} < \frac{5(1-c)^2 M^2}{d^2}$. Summing over all odd primes $d>5$ we must remove at most $5(1-c)^2 M^2 \sum_{d\geq 7 : \text{ prime}} d^{-2}< 5(1-c)^2 M^2$ \(\int_{6}^{\infty} 1/x^2 \,dx\) $=\frac{5}{6} (1-c)^2 M^2$. 

So $|S_P^+| \geq (1-c)^2 M^2 - \frac{5}{6} (1-c)^2 M^2 = \Omega(M^2)$. Thus $|S| \geq |S_P^+|/K = \Omega(M^2)$. Furthermore, $|S| \leq M^2$ so $|S|=\Theta\qty(M^2)$
\end{proof}

Therefore $I(P,L)= \Omega\qty(N^2 M^2)$.

\begin{lemma}
    Each line in $L$ has $\Theta\qty(\frac{N}{M})$ points in $P$.
\end{lemma}

\begin{proof}
Each line is $\Omega\qty(\frac{N}{M})$ rich: Let $l$ be an arbitrary line in $L$. There exist $s=\frac{p_1+p_1\sqrt{k}}{q_1+q_2\sqrt{k}} \in S$ and $(a,b) \in A_{N/4}^2$ such that $l$ is the line $y-b = s(x-a)$. Then for all $x = a + (q_1+\sqrt{k}q_2)(a_1+\sqrt{k}a_2)$ where $a_i \in \qty[0,\frac{\min_{i\in \{1,2 \}} \sqrt{N}-|b_i|}{(k+1)\sqrt{M}}]$ : $s(x-a)=(p_1+\sqrt{k}p_2)(a_1+\sqrt{k}a_2) = (p_1a_1+kp_2a_2) + (p_1a_2 + p_2a_1)\sqrt{k}$ where each of the linearly independent terms have integer coefficients in the range $\max_{i \in {1,2}} \qty[ -\sqrt{N} - b_i,\sqrt{N}-b_i]$. So  for all $\Omega\qty(\qty(\frac{\sqrt{N}}{\sqrt{M}})^2)$ choices of $a_1 ,a_2$ there exists $y \in A_N$ such that $y-b = s(x-a)$. Thus each line in $L$ has $\Omega\qty(\frac{N}{M})$ points in $P$.

Each line in $L$ has $O\qty(\frac{N}{M})$ points in $P$:
Let $(x,y) \in  A_N^2 = P$ such that $(x,y) \in l$ for some line $l : y = \frac{p_1+p_2\sqrt{k}}{q_1+q_2\sqrt{k}}(x-a) + b \in L$. Letting $Y_1 + Y_2\sqrt{k}= y - b$ this is equivalent to $x = \frac{(Y_1+Y_2\sqrt{k})(q_1+q_2\sqrt{k})}{p_1+p_2\sqrt{k}} + a$.

\begin{equation*}
    \implies x = a + \frac{(Y_1q_1p_1-kY_1q_2p_2-kY_2q_1p_2+kY_2q_2p_1)+(Y_2q_1p_1+Y_1q_2p_1-Y_1q_1p_2-kY_2q_2p_2)\sqrt{k}}{p_1^2-kp_2^2}
\end{equation*}

\begin{equation*}
    x \in A_N \implies \begin{cases}
    p_1^2 - kp_2^2 |q_1(Y_1p_1-kY_2p_2)+kq_2(Y_2p_1-Y_1p_2)\\
    p_1^2 - kp_2^2 |q_2(Y_1p_1-kY_2p_2)+q_1(Y_2p_1-Y_1p_2)
    \end{cases}
\end{equation*}

\begin{equation*}
    \implies
    p_1^2 - kp_2^2 | (q_1^2-kq_2^2)(Y_1p_1-kY_2p_2)
\end{equation*}

$\frac{p_1+p_2\sqrt{k}}{q_1+q_2\sqrt{k}} \in S$ so $\gcd(p_1^2+k p_2^2,q_1^2+kq_2^2) \leq 5$ Thus $p_1^2-kp_2^2 | 30(Y_1p_1-kY_2p_2)$. $Y_1,Y_2 \in \qty[-\sqrt{N},\sqrt{N}]$ and $|p_i| \in \qty[c\sqrt{M},\sqrt{M}]$ where $c<1$ is a constant that we choose to be sufficiently close to 1. Then $|p_1^2-kp_2^2|>|M-k(1-c)^2 M| = \Omega(M)$. So $ |\frac{Y_1p_1-kY_2p_2}{p_1^2-kp_2^2}| =O \qty( \sqrt{\frac{N}{M}})$. So $p_1^2 - kp_2^2 | 30(Y_1p_1-kY_2p_2)$ if and only if there exists and integer $|j|=O\qty(\sqrt{\frac{N}{M}})$ such that $j(p_1^2-kp_2^2)=30(Y_1p_1-kY_2p_2)$

\begin{equation*}
    \implies \begin{cases}
    Y_1=jp_1/30+\frac{kp_2(Y_2-p_2)}{30 p_1}\\
    Y_2=jp_2/30+\frac{p_1(Y_1-p_1)}{30 k p_2}
    \end{cases} \implies \begin{cases}
    p_1 | k( Y_2-p_2)\\
    kp_2 | Y_1-p_1
    \end{cases}
\end{equation*}

Since $\gcd(p_1,p_2) \leq 5$ and $Y_i, p_i$ are integers. Also $|p_i| \geq c\sqrt{M}$ and $Y_1 \equiv p_1 \mod{kp_2}$ and $Y_1 \in \qty[-\sqrt{N},\sqrt{N}]$ so there are $O(\sqrt{\frac{N}{M}})$ choices for $Y_1$. Plugging $Y_1$ into the system of equations above uniquely defines $Y_2$ so there are $O(\sqrt{\frac{N}{M}})$ choices of $(Y_1,Y_2)$ for each $j$. There are $O(\sqrt{\frac{N}{M}})$ choices of $j$ so each line is $O(\sqrt{\frac{N}{M}}\cdot\sqrt{\frac{N}{M}})=O(\frac{N}{M})$ rich.

\end{proof}

Each line in $L$ has $\Theta\qty(\frac{N}{M})$ points in $P$ so $I(P,L) = \Theta\qty(|L| \frac{N}{M})$. Combining with $I(P,L)= \Theta\qty(N^2 M^2)$ we obtain $|L| = \Theta\qty( N M^3)$.

The Szemerédi-Trotter bound for point set $P$ states that the number of $\frac{N}{M}$ rich lines is $O\big{(} \frac{(N^2)^2}{(N/M)^3} + \frac{N^2}{N/M} \big{)} = O\big{(} N M^3 \big{)}$. So we have achieved the Szemerédi-Trotter upper bound for any richness.
\end{proof}

\section{Applications} \label{sec:applications}

In this section we prove Lemma \ref{th:LM_lattice} and Corollary \ref{th:LM_GAP}, two sharp examples of the inverse discrete Loomis-Whitney problem in the plane. We use the family of constructions from Theorem \ref{th:construction} to show two lemmas bounding additive and multiplicative energies \cite{MRS13} are sharp.  We first recall the statements.
\vspace{2mm}

\noindent {\bf Corollary \ref{th:LM_GAP}.}$\qquad$\\
\emph{For any non-square integer $k$, any large enough $N$, let the point set $\pts=A_N^2$ where $A_N = \qty{ x_1 + x_2 \sqrt{k} ; x_1,x_2 \in \qty[-\sqrt{N},\sqrt{N}]}$. Then for any constant $p=O(1)$ there is a set of projections $\big\{ \pi_i \}_{i=1}^{\Theta(p)}$ such that $|\pi_i(\pts) | = \Theta \qty( \sqrt{p} n)$.
}

\begin{proof}
We see $\pts$ as embedded in the larger point set $\pts' = A_{4N}^2$. Letting $p=M^2$, we construct the set of $\frac{n}{\sqrt{p}}-$rich lines on $\pts'$ from the proof of Theorem \ref{th:construction}. These belong to $|S|$ many families of parallel lines each of size $\Theta(\sqrt{p}n)$, such that every point in $\pts$ is in exactly one line from every family. The size of the slope set is $|S|=\Theta(M^2)=\Theta(p)$. Letting $S$ be the projection directions, the size of each projection is equal to the number of lines in each family $= \Theta(\sqrt{p}n)$ .
\end{proof}

\noindent {\bf Lemma \ref{th:LM_lattice}.}$\qquad$\\
\emph{Let the point set $\pts$ be a section of the integer lattice of size $n \times n$. A one-dimensional projection  of $\pts$ has size $\Theta(n)$ if and only if the slope of the projection direction is an irreducible rational $p/q$ such that $p,q=O(1)$.
}

\begin{proof}
Any line whose slope is non-rational will go through at most a single point, so the projection of the point set along this direction will have maximal size of $n^2$. So projections of size $O(n)$ can only exist along rational projection directions. 

Furthermore we know from the proof of \ref{th:LM_GAP}, taking the case where $k$ is a square, so the point set reduces to the case of a square lattice, that if $p,q=O(1)$ then the projection along the slope $p/q$ has size $\Theta(n)$. 

If $p = \omega(1)$, $y=\frac{p}{q} \cdot x \in [0,n] \implies x/q = o(n)$. Similarly, if $q = \omega(1)$, $y=\frac{p}{q} \cdot x \in [0,n] \implies y/q = o(n)$. In either case there are asymptotically less than $n$ points of the lattice on each line of slope $\frac{p}{q}$. So the projection along $\frac{p}{q}$ has size $\omega(n)$.

\end{proof}

\noindent {\bf Lemma \ref{th:energy_ex}.}$\qquad$\\
\emph{For any non-square constant $k$, and $M \leq N$ let 
\\$A_N = \qty{ x_1 + x_2 \sqrt{k} ; x_1,x_2 \in \qty[-\sqrt{N},\sqrt{N}]}$ and let $X=S \subset \frac{A_N}{A_N}$ be the slope set from the proof of Theorem \ref{th:construction}. Then $|X| \leq |A_N|^2$ and
\[ \sum_{x \in X}  E^+(A_N,xA_N) = \Theta\qty( |A_N|^3 |X|^{1/2}). \]
}

\begin{proof}
We consider the dual situation of the proof of Lemma 2.3 \cite{MRS13}. I let the point set be $\pts = A_N^2$ and the line set to be as in the construction in the proof of Theorem \ref{th:construction}. Then  $\sum_{x \in X} E^+(A_N,xA_N) = \sum_{x \in X} \sum_{y} r_{A+Bx}^2 (y) = \Theta\qty(\sum_{\text{lines}} (\frac{N}{M})^2)$  since each line in the construction is $\frac{N}{M}$ rich. Furthermore there are $\Theta\qty(N\cdot M^3)$ lines in the construction so $\sum_{x \in X} E^+(A_N,xA_N) = \Theta(N^3\cdot M)$. Finally $|A_N|=N$ and $|X|=|S|=\Theta(M^2)$ \ref{th:slopes} so I have shown $\sum_{x \in X}  E^+(A_N,xA_N) = \Theta\qty( |A_N|^3 |X|^{1/2}).$
\end{proof}

\bibliographystyle{plain}
\bibliography{bibliography}

\begin{thebibliography}{10}

\bibitem{AB21}
David Alonso-Guti\'errez and Silouanos Brazitikos.
\newblock Reverse loomis-whitney inequalities via isotropicity.
\newblock {\em Proceedings of the American Mathematical Society}, 149:817--828,
  2021.

\bibitem{ABCP}
Gagik Amirkhanyan, Albert Bush, Ernest Croot, and Chris Pryby.
\newblock Sets of rich lines in general position.
\newblock {\em Journal of the London Mathematical Society}, 96(1):67--85, 2017.

\bibitem{Ball}
Keith~M. Ball.
\newblock Shadows of convex bodies.
\newblock {\em Transactions of the American Mathematical Society},
  327:891--901, 1991.

\bibitem{Beck83}
J{\'o}zsef Beck.
\newblock On the lattice property of the plane and some problems of dirac,
  motzkin and erd{\H{o}}s in combinatorial geometry.
\newblock {\em Combinatorica}, 3(3-4):281--297, 1983.

\bibitem{BCW}
Jonathan Bennett, Anthony~P. Carbery, and James Wright.
\newblock A non-linear generalization of the loomis–whitney inequality and
  applications.
\newblock {\em Mathematical Reseach Letters}, 12:443--457, 2005.

\bibitem{BCT06}
Jonathan Bennett, Anthony~P. Carbery, and James Wright.
\newblock On the multilinear restriction and kakeya conjectures.
\newblock {\em Acta Mathematica}, 196:261--302, 2006.

\bibitem{BT}
B\'ela Bollob\'as and Andrew Thomason.
\newblock Projections of bodies and hereditary properties of hypergraphs.
\newblock {\em Duke Mathematical Journal}, 59(2):337--357, 1989.

\bibitem{bombieri2015problem}
Enrico Bombieri and Jean Bourgain.
\newblock A problem on sums of two squares.
\newblock {\em International Mathematics Research Notices},
  2015(11):3343--3407, 2015.

\bibitem{BD15}
Jean Bourgain and Ciprian Demeter.
\newblock New bounds for the discrete fourier restriction to the sphere in 4d
  and 5d.
\newblock {\em International Mathematics Research Notices},
  2015(11):3150--3184, 2015.

\bibitem{CGG18}
Stefano Campi, Peter Gritzmann, and Paolo Gronchi.
\newblock On the reverse loomis--whitney inequality.
\newblock {\em Discrete \& Computational Geometry}, 60(1):115--144, 2018.

\bibitem{CGS16}
Artem Chernikov, David Galvin, and Sergei Starchenko.
\newblock Cutting lemma and zarankiewicz’s problem in distal structures.
\newblock {\em Selecta Mathematica}, 26(2):1--27, 2020.

\bibitem{CL}
Ernie Croot and Vsevolod~F Lev.
\newblock Open problems in additive combinatorics.
\newblock {\em Additive Combinatorics}, 43:207--233, 2007.

\bibitem{Do20}
Thao Do.
\newblock Extending erdős–beck's theorem to higher dimensions.
\newblock {\em Computational Geometry}, 90, 2020.

\bibitem{Elekes97b}
Gy{\"o}rgy Elekes.
\newblock On linear combinatorics i. concurrency—an algebraic approach.
\newblock {\em Combinatorica}, 17(4):447--458, 1997.

\bibitem{Elek02}
Gy{\"o}rgy Elekes.
\newblock Sums versus products in number theory, algebra and erdos geometry.
\newblock {\em Paul Erd\H{o}s and his Mathematics II}, 11:241--290, 2001.

\bibitem{gromov}
Mikhael Gromov.
\newblock Entropy and isoperimetry for linear and non-linear group actions.
\newblock {\em Group, Geometry, and Dynamics}, 2:499--593, 2008.

\bibitem{katz2019improved}
Nets Katz and Joshua Zahl.
\newblock An improved bound on the hausdorff dimension of besicovitch sets in
  r$^3$.
\newblock {\em Journal of the American Mathematical Society}, 32(1):195--259,
  2019.

\bibitem{LW}
Lynn~H. Loomis and Hassler Whitney.
\newblock An inequality related to the isoperimetric inequality.
\newblock {\em Bulletin of the American Mathematical Society}, 55:961--962,
  1949.

\bibitem{Murphy}
Brendan Murphy.
\newblock Upper and lower bounds for rich lines in grids.
\newblock {\em American Journal of Mathematics}, 143, 2021.

\bibitem{MRS13}
Brendan Murphy, Oliver Roche-Newton, and Ilya~D. Shkredov.
\newblock Variations on the sum-product problem.
\newblock {\em SIAM Journal on Discrete Mathematics}, 29:514--540, 2013.

\bibitem{ShefferSilier}
Adam Sheffer and Olivine Silier.
\newblock A structural szemerédi-trotter theorem for cartesian products.
\newblock {\em arXiv preprint arXiv:2110.09692}, 2021.

\bibitem{singer2021point}
Noah Singer and Madhu Sudan.
\newblock Point-hyperplane incidence geometry and the log-rank conjecture.
\newblock {\em arXiv preprint arXiv:2101.09592}, 2021.

\bibitem{Solymosi06}
J{\'o}zsef Solymosi.
\newblock Dense arrangements are locally very dense. i.
\newblock {\em SIAM Journal on Discrete Mathematics}, 20(3):623--627, 2006.

\bibitem{ST83}
Endre Szemer{\'e}di and William~T. Trotter.
\newblock Extremal problems in discrete geometry.
\newblock {\em Combinatorica}, 3(3):381--392, 1983.

\end{thebibliography}

\end{document}